\documentclass[review]{elsarticle}
\usepackage[usenames]{color}
\usepackage{extarrows}
\usepackage{lineno,hyperref}
\modulolinenumbers[5]

\journal{Journal of \LaTeX\ Templates}

\usepackage{amsfonts}
\usepackage{amsmath, cases}
\usepackage{indentfirst}
\usepackage{graphicx}
\usepackage{latexsym,bm, amsthm}

\usepackage{epsfig}
\usepackage{latexsym}
\usepackage{amssymb}
\newcommand{\R}{\mathcal{R}}
\newcommand{\N}{\mathcal{N}}
\newcommand{\C}{\mathbb{C}}
\newcommand{\F}{{\mathbf{F}}}
\newcommand{\G}{{\mathbf{G}}}

\begin{document}

\begin{frontmatter}

\title{
 The core inverse and constrained matrix approximation problem }

\author[]%[1-address]
{Hongxing Wang\corref{mycorrespondingauthor}}\cortext[mycorrespondingauthor]{Corresponding author}
\ead{winghongxing0902@163.com}

 \author[]%[1-address]
 {Xiaoyan Zhang}
\ead{1310362243@qq.com}
\address%[1-address]
{School of Science,
Guangxi Key Laboratory of Hybrid Computation and IC Design Analysis,
Guangxi University for Nationalities,
Nanning 530006, China }

\begin{abstract}
In this paper,we study the constrained matrix approximation problem in the Frobenius norm by using the core inverse:\begin{align}\nonumber \left\|{Mx - b} \right\|_F=\min \mbox{\rm \ \ subject to \ \ }{x\in\R(M)} ,\end{align}where $M\in\mathbb{C}^{\mbox{\rm\footnotesize\texttt{CM}}}_{n }$. We get the unique solution to the  problem, provide two Cramer's rules for the unique solution, and establish two new expressions  for the core inverse.
\end{abstract}


\end{frontmatter}

\linenumbers

\section{Introduction}
\numberwithin{equation}{section}
\newtheorem{theorem}{T{\scriptsize HEOREM}}[section]
\newtheorem{lemma}[theorem]{L{\scriptsize  EMMA}}
\newtheorem{corollary}[theorem]{C{\scriptsize OROLLARY}}
\newtheorem{proposition}[theorem]{P{\scriptsize ROPOSITION}}
\newtheorem{remark}{R{\scriptsize  EMARK}}[section]
\newtheorem{definition}{D{\scriptsize  EFINITION}}[section]
\newtheorem{algorithm}{A{\scriptsize  LGORITHM}}[section]
\newtheorem{example}{E{\scriptsize  XAMPLE}}[section]
\newtheorem{problem}{P{\scriptsize  ROBLEM}}[section]
\newtheorem{assumption}{A{\scriptsize  SSUMPTION}}[section]

\label{Introduction}
 Let $M^\ast  $, ${\mathcal{R}}(M)$ and
 $\mathcal{N}(M)$
 stand for
 the
 \emph{conjugate transpose},
\emph{range space} %(or \emph{column space})
\emph{null space} %(or \emph{column space})
 of $M\in {\C^{m \times n}}$,
   respectively.
 The symbol
$M(i \rightarrow b)$ denotes a matrix   from $M$
by replacing the $i$-th column of $M$ by %a vector
$b\in {\C^{ n}}$.
 The symbol $e_i$ denotes the $i$-th column of $I_n$ in which
$1 \leq i \leq n$.
%Then
%$Me_i$ is
%the $i$-th column of $M$.
%
%
%
%
%
%
%
The \emph{Moore-Penrose inverse} of $M %\in\mathbb{C}^{m\times n}
$ is
    the unique matrix $X\in\mathbb{C}^{n\times m}$ satisfying  the relations:
 $MXM = M$,
 $XMX = X$,
 $\left( {MX} \right)^\ast = MX$
 and
 $\left( {XM} \right)^\ast = XM$,
 and is  denoted by $X = M^{\dag}$,
 ({see \rm\cite{Ben2003book,Dragana2017book,Wang2018book}}).

Let $M\in\mathbb{C}^{n\times n}$ be singular.
The smallest positive integer $k$ for which
${\rm rk}\left( M^{k+1} \right)
=
{\rm rk}\left( M^k \right)$
is called the \emph{index} of $M$ and is  denoted by  ${\rm Ind}(M)$.
 The index of a non-singular matrix
is $0$ and the index of a null matrix is $1$.
Furthermore,
\begin{align}
\label{EP-1}
\mathbb{ C}^{\mbox{\rm\footnotesize\texttt{CM}}}_n &
=\left\{M\left|\
{\rm Ind}(M) \leq  1,\
   M\in \mathbb{C}^{n\times n} \right.\right\}.
\end{align}

Let
$M\in\mathbb{C}^{n\times n}$
with
 ${\rm Ind}(M)=k$.
A matrix
$X$ is
the \emph{Drazin inverse} of $M$ if %with
$MXM^k = M^k$,
$XMX = X$
and
$MX = XM$.
We write
$X = M^D$ for the  Drazin inverse of $M$.
In  particular,
 when  $M\in \mathbb{C}^{\mbox{\rm\footnotesize\texttt{CM}}}_{n}$,
 the matrix  $X$   is  the \emph{group inverse} of $M$,
 and is denoted by $X = M^\#$,
 ({see \rm\cite{Ben2003book,Dragana2017book,Wang2018book}}).

The \emph{core inverse}
 of
$M \in {\mathbb{C} }_{n}^{\mbox{\rm\footnotesize\texttt{CM}}}$
is defined as the unique matrix $X\in\mathbb{C}^{n\times n}$
satisfying   the equations:
$MXM = M$, $MX^2 = X$ and $\left( {MX} \right)^\ast = MX$,
and is denoted by $X = M^{\tiny\textcircled{\#}}$,
({see \rm\cite{Baksalary2010lma681,Wang2015lma1829}}).
It is noteworthy that
the core inverse is a ``least squares''  inverse,
(see \cite{Cline1968siamjna182,Malik2014amc575}).
Moreover,
it is proved that
 $ M^{\tiny\textcircled{\#}}
 =M^\#MM^\dag$(see \cite{Baksalary2010lma681}).

Recently,
the relevant conclusions of the core inverse are very rich.
In  \cite{Baksalary2014amc450,Malik2014amc575,
Manjunatha2014lma792,Wang2018OP1218%,Mehdipour2018lma1046
},
generalizations of core inverse are introduced,
for example,
the core-EP inverse and the weak group inverse, etc.
In \cite{Ma2018amc176,Prasad2018sm193,
Wang2016laa289,Wang2018jsu135}, %,Zhou2018amc187
their algebraic properties
and
calculating methods
are studied.
In  \cite{Gao2018ca38,Dragan2014amc283%,Yu2016c331%,Zou2018mjm15
},
the studying of them
is  extended to some new fields, for example, ring and operator, etc.
Moreover,
 the inverses are used to study partial orders
in \cite{Baksalary2010lma681,
Wang2018OP1218,
Wang2015lma1829,
Wang2018lma206}.

\bigskip

  Consider the following equation:
\begin{align}
\label{CM-0}
&Mx = b.%
\end{align}
Let $M\in\C^{n\times n}$
with   ${\rm{Ind}}(M) = k$,
and
$b\in\R\left(M^k\right)$. %1979ÄêµÄ[P.123]
Campbell and Meyer\cite{Campbell2009book}
 show that
$x=M^Db$ is the unique solution of (\ref{CM-0})
with respect to $x\in\R\left(M^k\right)$.
  Wei\cite{Wei1998amc115} gets the minimal $P$-norm solution
  of (\ref{CM-0}),
  where
  $P$ is nonsingular,
  ${P^{ - 1}}MP$   is the Jordan
canonical form of $M$
and
${\left\| x \right\|_p}={\left\| {{P^{ - 1}}x}\right\|_2}$.
Furthermore,
let $M \in {\C^{m \times n}}$.
   Wei\cite{Wei2002amc303} considers
the unique solution of
\begin{align*}
WMWx = b
\quad \mbox{\rm subject to} \quad
x \in \mathcal{R}\left( {{{\left( {MW} \right)}^{{k_1}}}} \right),
\end{align*}
  where $W \in {\mathbb{C}^{n \times m}}$,
  $k_1={\rm{Ind}}\left( {MW} \right)$,
 $k_2={\rm{Ind}}\left( {WM} \right)$
and
$b \in \mathcal{R}\left( {{{\left( {WM} \right)}^{{k_2}}}} \right)$.
More results
of (\ref{CM-0})
  under some certain conditions can be found in
 \cite{Chen1993lma339,Morikuni2018siam1033,
 Toutounian2017amc343,Wang1989laa27,
 Wang2018book,Wei1998amc115,
 Wei2000jcam305}.

It is  well-known that
$b\in\R\left(M\right)$
if and only if
(\ref{CM-0}) is solvable.
Let
$b\in\R\left(M\right)$
and
the index of $M$ is 1,
then $x=M^\#b$ is the unique solution with
$b\in\R\left(M\right)$,
 \cite{Campbell2009book}.
It follows from
$ M^{\tiny\textcircled{\#}}
 =M^\#MM^\dag$
 that $M^\#b=M^{\tiny\textcircled{\#}}b$,
 \cite{Ma2019lmaonline}.
Furthermore,
the unique solution $x=M^{\tiny\textcircled{\#}}b$ is given by the Cramer's rule
\cite[Theorem 3.3]{Ma2019lmaonline}.

When $b\notin \R(M)$,
(\ref{CM-0})  is unsolvable, yet,
it has  least-squares solutions.
Motivated by the above mentioned work,
it is naturally to consider %introduce
%
%consider
the least-squares solutions of (\ref{CM-0})
under the  certain
condition $x\in\R(M)$,
i.e.,
\begin{align}
\label{CM-1}
\|Mx-b\|_F= \min \mbox{\rm \ \ subject to \ \ }{x\in\R(M)},
\end{align}
where $M \in \mathbb{C}^{\mbox{\rm\footnotesize\texttt{CM}}}_{n }$,
${\rm rk}(M)=r < n$
and
$b \in \mathbb{C}^{n }$.

\section{Preliminaries}
\label{Preliminaries}

\begin{lemma}
{\rm (\cite%[Theorem 8]
{Ben2003book})}
\label{CM-Lemma1-4}
Let $M \in \mathbb{C}^{n\times n}$ be idempotent.
Then
$M = P_{\R(M), \N(M)}$ with
$\R\left(M\right)\oplus  \N\left(M\right) =\mathbb{C}^n$.
On the contrary,
if
$\F\oplus  \G =\mathbb{ C}^n$,
then there exists an
idempotent $P_{\F, \G}$ such that
$\R\left(P_{\F, \G}\right)=\F$
and
$\N\left(P_{\F, \G}\right)=\G$.

Furthermore, $I-P_{\F, \G}=P_{\G, \F}$.
\end{lemma}

\begin{lemma}
{\rm(\cite{Wang2018book})}
\label{CM-Lemma1-5}
Let $M \in \mathbb{C}^{n\times n}$.
Then $\mathrm{Ind}(M) = k$ if and only if
\begin{align}
\R\left(M^k\right)\oplus  \N\left(M^k\right) =\mathbb{ C}^n.
\end{align}
\end{lemma}

\begin{lemma} {\rm(\cite{Wang2018book})}
\label{CM-Lemma1-3}
Let $MXM=M$ and $XMX=X$. %
Then\begin{align}
\nonumber
X M = P_{\R(X), \N(M)}
\mbox{\  and \ }
MX = P_{{ \R}(M), \N(X) }.
\end{align}
\end{lemma}

\begin{lemma} {\rm(\cite{Wang2018book})}
\label{CM-Lemma1-6}
Let
$\F \oplus  \G =\mathbb{ C}^n$.
Then
\begin{itemize}
  \item[{\rm(1)}]
  $P_{\F,\G}M=M\Leftrightarrow \R(M)\subseteq \F$;

  \item[{\rm(2)}]
  $MP_{\F,\G}=M\Leftrightarrow \N(M)\supseteq \G$.
\end{itemize}
\end{lemma}

 \begin{lemma}
{\rm(\cite{Baksalary2010lma681,Wang2016laa289})}
 Let
$M\in\mathbb{C}^{\mbox{\rm\footnotesize\texttt{CM}}}_{n }$
  with
${\rm rk}(M)=r$.
Then
there exists  a   unitary matrix $V$
 such that
\begin{align}
\label{CM-10}
M = V\left[ {{\begin{matrix}
 T   & S   \\
 0   & 0   \\
\end{matrix} }} \right]V^\ast,
\end{align}
where $T$ is nonsingular.
Furthermore,
\begin{align}
\label{CM-11}
M^{\tiny\textcircled{\#}}
&
=
 V\left[ {{\begin{matrix}
 {T ^{ - 1}}   & 0   \\
 0   & 0   \\
\end{matrix} }} \right]V^\ast .
\end{align}
\end{lemma}

\section{Main Results}
\label{Main-Results}
\begin{theorem}
\label{CM-Theorem-1}
Let $M \in \mathbb{C}^{\mbox{\rm\footnotesize\texttt{CM}}}_{n }$
and
$b \in \mathbb{C}^{n }$.
Then
\begin{align}
\label{CM-00}
x=
 {M^{\tiny\textcircled{\#}}}b
\end{align}
is the unique solution of (\ref{CM-1}).
\end{theorem}
\begin{proof}
From
$x \in \R\left( M \right)$,
it follows  that there exists $y \in {\C^n}$ for which $x = My$.
Let the decomposition of $M$ be as in (\ref{CM-10}).
Denote
\begin{align}
{V^ * }y = \left[ {\begin{matrix}
{{y_1}}\\
{{y_2}}
\end{matrix}} \right],
\
{V^ * }b = \left[ {\begin{matrix}
{{b_1}}\\
{{b_2}}
\end{matrix}} \right]
\
\mbox{\rm and }
\
 {M^ {\tiny\textcircled{\#}} }b = V\left[ {\begin{matrix}
{{T^{ - 1}}{b_1}}\\
0
\end{matrix}} \right],
\end{align}
where $y_1$,
${b_1}$ and ${{T^{ - 1}}{b_1}}\in\C^{{\rm rk}(M)}$.
It follows that
\begin{align}
\nonumber
\left\| {Mx - b} \right\|_F^2
&
 =
  \left\| {\left[
{\begin{array}{*{20}{c}}
{{T^2}{y_1} + TS{y_2} - {b_1}}\\
{{-b_2}}
\end{array}} \right]} \right\|_F^2
\\
\nonumber
&
=
\left\| {{T^2}{y_1} + TS{y_2} - {b_1}} \right\|_F^2 + \left\| {{b_2}}
\right\|_F^2.
\end{align}
Since $T$ is invertible,
we have
$\mathop {\min }\limits_{{y_1},{y_2}} \left\|
{{T^2}{y_1} + TS{y_2} - {b_1}} \right\|_F^2 = 0$,
when
\begin{align*}
{y_1} = {T^{ - 2}}{b_1} - {T^{ - 1}}S{y_2}.
\end{align*}
Therefore,
\begin{align*}
x
&
 =
  My
  =
  V\left[ {\begin{matrix}
T&S\\
0&0
\end{matrix}} \right]{V^ * }y
=
V\left[ {\begin{matrix}
{T{y_1} + S{y_2}}\\
0
\end{matrix}} \right]
\\
&
=
 V\left[ {\begin{matrix}
{{T^{ - 1}}{b_1}}\\
0
\end{matrix}} \right]
 =
 {M^ {\tiny\textcircled{\#}} }b,
\end{align*}
that is,
(\ref{CM-00})
is  the unique solution of
(\ref{CM-1}).
\end{proof}

\bigskip

{
When $M\in \C^{n\times n}$ is nonsingular,
it is  well-known that
the solution of (\ref{CM-0}) is unique and
$x=M^{-1}b$.
Let $x=\left(x_1,x_2,...,x_n\right)^T$.
Then
 \begin{align}
 \label{CM-16}
{x_i}
=\frac{{\det \left(M\left( {i \to b} \right) \right)} }{\det (M)} ,
\quad
i = 1,2, \ldots ,n,
\end{align}
is called the Cramer's rule for solving (\ref{CM-0}).
In \cite{Ben1982laa223},
Ben-Israel
gets a  Cramer's rule
for obtaining the least-norm solution of the consistent linear system (\ref{CM-0}),
 \begin{align}
 \nonumber
{x_i}
=\frac{{\det \left(\left[\begin{matrix}
M\left( {i \to b} \right) & U\\
V^\ast(i \to 0) &0
\end{matrix}\right]\right)} }{\det
\left(\left[\begin{matrix}
M & U\\
V^\ast &0
\end{matrix}\right]\right)},
\quad
i = 1,2, \ldots ,n,
\end{align}
where
$U$ are $V$ are of full column rank,
 $\mathcal{R}\left(U\right)=\mathcal{N}\left(M^\ast\right)$
 and
 $\mathcal{R}\left(V\right)=\mathcal{N}\left(M\right)$.
In \cite{Wang1989laa27},
  Wang gives a Cramer's rule for
the unique solution $x \in \mathcal{R}\left( {{M^k}} \right)$ of (\ref{CM-0}),
where $b\in\R\left(M^k\right)$ and  ${\rm{Ind}}(M) = k$.
In \cite{Ji2005laa183},
Ji proposes two
new condensed Cramer's rules for the
unique solution $x \in \mathcal{R}\left( {{M^k}} \right)$ of (\ref{CM-0}),
where $b\in\R\left(M^k\right)$ and  ${\rm{Ind}}(M) = k$.
Furthermore,
in \cite{Ji2012laa2173},
Ji obtains
a new condensed Cramer's rule of Werner
for minimal-norm least-squares solution of (\ref{CM-0}).
More details
of Cramer's rules for finding restricted solutions of (\ref{CM-0})
can be found in
 \cite{Ben2003book,
 Kyrchei2012amc6375,Kyrchei2013asmc7632,Kyrchei2015book,
 %Kyrchei2012amc6375,Kyrchei2013laa136,
 Kyrchei2017amc1,%Ji2005laa183,Ji2012laa2173,
% Song2018amc490,Wang2004amc333,
 Wang2018book,Wang2005amc329%,Werner1984lma319
 }.
In the following
Theorem \ref{CM-Theorem-3}
and
Theorem \ref{CM-Theorem-4},
we will give two Cramer's rules for
the unique solution of (\ref{CM-1}).
}

First of all,
we give the following two lemmas
to prepare for
a  Cramer's ruler
for the core inverse in Theorem \ref{CM-Theorem-3}.

\begin{lemma}
\label{CM-Lemma-M-2}
Let
$M $ be as in (\ref{CM-10}),
and let
$L\in \mathbb{C}^{n \times \left( {n - r} \right)}$
with
${\rm rk}(L)={n - r}$
and
$\R\left(L\right)=\N\left( {{M^\ast}} \right)$.
Then
\begin{align}
\label{CM-9}
M^{\tiny\textcircled{\#}}M
+
\left(I_n-M^{\tiny\textcircled{\#}}M \right)
L\left(L^\ast L\right)^{-1}L^\ast
 =
 I_{n}.
\end{align}
\end{lemma}
\begin{proof}
Let
$M $ be as in (\ref{CM-10}),
applying Lemma \ref{CM-Lemma1-5},
we see that
\begin{align}
\label{CM-7}
\R\left(M \right)\oplus  \N\left(M \right) =\mathbb{ C}^n.
\end{align}

Denote
$M_1=I_n-M^{\tiny\textcircled{\#}}M $
and
$M_2=L\left(L^\ast L\right)^{-1}L^\ast $.

Applying
Lemma \ref{CM-Lemma1-4},
Lemma \ref{CM-Lemma1-3} and
$M^{\tiny\textcircled{\#}}M=M^\#M$,
we have
\begin{align}
\label{CM-8}
M^{\tiny\textcircled{\#}}M &= P_{\R(M), \N(M )},
\\
\label{CM-03}
M_1=I-M^{\tiny\textcircled{\#}}M &= P_{\N(M), \R(M )}.
\end{align}
Since
$\left(L\left(L^\ast L\right)^{-1}\right)
L^\ast
\left(L\left(L^\ast L\right)^{-1}\right)
=L\left(L^\ast L\right)^{-1}$
and
$L^\ast
\left(L\left(L^\ast L\right)^{-1}\right)
L^\ast
=L^\ast$,
applying
Lemma \ref{CM-Lemma1-3},
we obtain
\begin{align}
\label{CM-04}
M_2
&
= P_{\R(L), \N\left(L^\ast\right)}
= P_{\N(M), \R(M)}.
\end{align}
Since
$\R\left(L\right)=\N\left( {{M^\ast}} \right)$,
we obtain
$M_1M_2=M_2$, $M_2^2=M_2$,
\begin{align}
\label{CM-01}
M_2M_1M_2
&
=M_2
\mbox{ \  and  \ }
M_1M_2M_1
=M_1.
\end{align}
Using  Lemma \ref{CM-Lemma1-3}
to (\ref{CM-01}),
we have
\begin{align}
\label{CM-05}
M_1M_2
=
 P_{\R\left(M_1 \right),
 \N\left(M_2\right)}.
\end{align}
Applying (\ref{CM-03}) and (\ref{CM-04})
to (\ref{CM-05}),
we obtain
\begin{align}
\label{CM-02}
 M_1M_2&
= P_{\N(M), \R(M)}.
\end{align}

Therefore,
applying Lemma \ref{CM-Lemma1-4}, (\ref{CM-8}) %, (\ref{CM-05})
and (\ref{CM-02}),
we gain
\begin{align}\nonumber
M^{\tiny\textcircled{\#}}M
+
 M_1M_2&
=P_{\R(M), \N(M )}+ P_{\N(M), \R(M)}=I_n,
\end{align}
i.e., (\ref{CM-9}).
\end{proof}

\bigskip

{
In \cite[Theorem 3.2 and Theorem 3.3]{Ma2019lmaonline},
let $M \in \mathbb{C}^{\mbox{\rm\footnotesize\texttt{CM}}}_{n }$,
$b \in \mathbb{C}^{n }$
and
$b\in \mathcal{R}(M)$,
and
let $M_b$ and $M_c$ be of the full column ranks with
$\mathcal{N}(M^\ast)=\mathcal{R}(M_b)$
and
$\mathcal{N}(M_c^\ast)=\mathcal{R}(M)$.
Then
\begin{align*}
\left[ {\begin{matrix}
M      &M_b\\
M_c^\ast &0
\end{matrix}} \right]
\end{align*}
is invertible
and
the unique solution $x=^{\tiny\textcircled{\#}}b$
of (\ref{CM-0})
satisfying
\begin{align*}
{x_i} = {{\det \left(\left[ {\begin{matrix}
{M\left( {i \to b} \right)}&M_b\\
{M_c^\ast\left( {i \to 0} \right)}&0
\end{matrix}}\right] \right)} \mathord{\left/
{\det \left(\left[ {\begin{matrix}
M      &M_b\\
M_c^\ast &0
\end{matrix}} \right]\right)}\right.}},
\end{align*}
where $i = 1,2, \ldots ,n$.
In the following Lemma \ref{CM-lemma-M-1}
and Theorem \ref{CM-Theorem-3},
we give the unique least-squares solution of  (\ref{CM-1})
in a similar way under weaker conditions.
}

\begin{lemma}
\label{CM-lemma-M-1}
Let $M$ and $L$ be as in Lemma \ref{CM-Lemma-M-2}.
Then  \begin{align}
\label{CM-3}
G
= %\triangleq
\left[ {\begin{matrix}
M      &L\\
L^\ast &0
\end{matrix}} \right]
\end{align}
 is invertible
and
\begin{align}
\label{CM-4}
G^{-1}
=
 \left[ {\begin{matrix}
M^{\tiny\textcircled{\#}}
&
\left(I_n-M^{\tiny\textcircled{\#}}M \right)
L\left(L^\ast L\right)^{-1}  \\
\left(L^\ast L\right)^{-1}  L^\ast    &0
\end{matrix}} \right].
\end{align}
\end{lemma}
\begin{proof}
Since
$\R\left(L\right)=\N\left( {{M^\ast}} \right)$,
we have
$M^{\tiny\textcircled{\#}}   L=0$
and
$\left(L^\ast L\right)^{-1}L^\ast M=0$.
Furthermore,
applying (\ref{CM-9}),
we have
\begin{align*}
&
 \left[ {\begin{matrix}
M^{\tiny\textcircled{\#}}
&
\left(I_n-M^{\tiny\textcircled{\#}}M \right)
L\left(L^\ast L\right)^{-1}
 \\
\left(L^\ast L\right)^{-1}  L^\ast    &0
\end{matrix}} \right]
\left[ {\begin{matrix}
M      &L\\
L^\ast &0
\end{matrix}} \right]\\
&
=
  \left[ {\begin{matrix}
M^{\tiny\textcircled{\#}}M
+
\left(I_n-M^{\tiny\textcircled{\#}}M \right)
L\left(L^\ast L\right)^{-1}L^\ast
&
M^{\tiny\textcircled{\#}}   L
 \\
 \left(L^\ast L\right)^{-1}L^\ast M
 &
 \left(L^\ast L\right)^{-1}L^\ast   L
\end{matrix}} \right]
 \\&
=
I_{2n-r},
\end{align*}
that is,
$G$ is invertible and $G^{-1}$ is of the form (\ref{CM-4}).
\end{proof}

Based on Lemma \ref{CM-Lemma-M-2}
and
Lemma \ref{CM-lemma-M-1},
we get a Cramer's rule for the unique solution of (\ref{CM-1}).
\begin{theorem}
\label{CM-Theorem-3}
Let $M$ and $b$ be as in (\ref{CM-10}),
and
let  $L$  be as in Lemma \ref{CM-Lemma-M-2}.
Then (\ref{CM-1}) has the unique solution
$x =\left( {{x_1},{x_2}, \ldots ,{x_n}} \right)^T$
satisfying
\begin{align}
\label{CM-5}
{x_i} = {{\det \left(\left[ {\begin{matrix}
{M\left( {i \to b} \right)}&L\\
{L^\ast\left( {i \to 0} \right)}&0
\end{matrix}}\right] \right)} \mathord{\left/
{\det \left(\left[ {\begin{matrix}
M      &L\\
L^\ast &0
\end{matrix}} \right]\right)}\right.}},
\end{align}
where $i = 1,2, \ldots ,n$.
\end{theorem}
\begin{proof}
Since $G$ is invertible,
applying Lemma \ref{CM-lemma-M-1},
we get the unique solution
$\hat{x}=G^{-1}\hat{b}$
of $G\hat{x}=\hat{b}$,
in which
$\hat{x}^\ast
=
[\begin{matrix}
x^\ast &
y^\ast
\end{matrix}]^\ast$
and
$\hat{b}^\ast
=
[\begin{matrix}
b^\ast &
0
\end{matrix}]^\ast$.
It follows from
 (\ref{CM-4})
that
\begin{align*}
\left[\begin{matrix}
x\\
y
\end{matrix}\right]
&
=
\left[ {\begin{matrix}
M^{\tiny\textcircled{\#}}
&
\left(I_n-M^{\tiny\textcircled{\#}}M \right)
L\left(L^\ast L\right)^{-1}  \\
\left(L^\ast L\right)^{-1}  L^\ast    &0
\end{matrix}} \right]
\left[\begin{matrix}
b\\
0
\end{matrix}\right]
=
\left[\begin{matrix}
M^{\tiny\textcircled{\#}}b\\
\left(L^\ast L\right)^{-1}  L^\ast b
\end{matrix}\right].
\end{align*}
Applying
(\ref{CM-16})
we obtain (\ref{CM-5}).
\end{proof}

\bigskip

In the following theorem,
we give a characterization of the core inverse
and
prepare for
a  Cramer's ruler
for the core inverse in Theorem \ref{CM-Theorem-4}.

\begin{theorem}
\label{CM-Theorem-2}
Let $M$ and $L$   be as in (\ref{CM-3}).
Then
\begin{align}
\label{CM-000}
 {M^{\tiny\textcircled{\#}} }
 =
 {\left( {M{M^\ast  }M + L{L^\ast  }} \right)^{ - 1}}M{M^\ast}.
 \end{align}
\end{theorem}
\begin{proof}
Since
$\R\left(L\right)=\N\left( {{M^\ast}} \right)$,
$M\in\mathbb{C}^{\mbox{\rm\footnotesize\texttt{CM}}}_{n }$
and
$\R\left(M\right)\oplus  \N\left(M\right) =\mathbb{ C}^n$,
we obtain
\begin{align}
\nonumber
\left( {L{L^\ast  }} \right){\left( {L{L^\ast  }} \right)^\dag }
&
=
 {P_{\N\left( M \right),\R\left( M \right)}},
\\
\nonumber
\left( {M{M^\ast  }M} \right){\left( {M{M^\ast  }M} \right)^{\tiny\textcircled{\#}} }
&
= {P_{\R\left( M \right),\N\left(M \right)}}
\end{align}
and
\begin{align}
\nonumber
\left( {M{M^\ast  }M + L{L^\ast  }} \right)
\left( {{{\left( {M{M^\ast  }M} \right)}^{\tiny\textcircled{\#}} }
+ {{\left( {L{L^\ast  }} \right)}^\dag }}\right)
&
=
\left( {M{M^\ast  }M} \right){\left( {M{M^\ast  }M} \right)^{\tiny\textcircled{\#}} }
+
\left({L{L^\ast  }} \right){\left( {L{L^\ast  }} \right)^\dag }
\\
&
\nonumber
= {P_{\R\left( M \right),\N\left( M \right)}} + {P_{\N\left( M \right),\R\left( M
\right)}} = {I_n}.
\end{align}
Therefore,
$M{M^\ast  }M + L{L^\ast  }$ is invertible.

Since
$
{\left( {L{L^\ast  }} \right)^\dag }M{M^\ast  } = 0$
 and
${\left( {M{M^\ast  }M} \right)^{\tiny\textcircled{\#}} }M{M^\ast  }
 =
 {M^{\tiny\textcircled{\#}} }$,
we have
\begin{align}
\nonumber
 {\left( {M{M^\ast  }M + L{L^\ast  }} \right)^{ - 1}}M{M^\ast  }
=
{\left( {M{M^\ast  }M} \right)^{\tiny\textcircled{\#}} }M{M^\ast  }
+
{\left( {L{L^\ast  }} \right)^\dag }M{M^\ast  }
=
 {M^{\tiny\textcircled{\#}} }.
\end{align}
It follows
that (\ref{CM-000}).
\end{proof}

\bigskip

\begin{theorem}
\label{CM-Theorem-4}
Let $M$ and $L$   be as in (\ref{CM-3}).
Then (\ref{CM-1}) has the unique solution
$x =\left( {{x_1},{x_2}, \ldots ,{x_n}} \right)^T$
satisfying
\begin{align}
\label{CM-5-1}
{x_j} =
\frac{\det {\left( {M{M^\ast  }M + L{L^\ast  }} \right)}
\left( {j \to M{M^\ast  }b} \right) }
{\det {\left( {M{M^\ast  }M + L{L^\ast  }} \right)}} ,
\end{align}
where $j = 1,2, \ldots ,n$.
\end{theorem}
\begin{proof}
Applying
Theorem \ref{CM-Theorem-2}
to
Theorem \ref{CM-Theorem-1}, we have
\begin{align}
\nonumber
x=
 {\left( {M{M^\ast  }M + L{L^\ast  }} \right)^{ - 1}}M{M^\ast}b,
 \end{align}
that is,
\begin{align}
\nonumber
\left( {M{M^\ast  }M + L{L^\ast  }} \right)x=
M{M^\ast}b.
\end{align}
It follows from  (\ref{CM-16}) that we get (\ref{CM-5-1}).
\end{proof}

In \cite{Ji2005laa183},
Ji obtains  the condensed determinantal expressions of $M^\dag$ and $M^D$.
By using Theorem \ref{CM-Theorem-2},
we get a condensed determinantal expression of $M^{\tiny\textcircled{\#}}$.

\begin{theorem}
\label{CM-Theorem-5}
Let $M$ and $L$ be defined as in (\ref{CM-3}).
Then
the core inverse $M^{\tiny\textcircled{\#}}$ is given:
\begin{align}
\label{CM-5-2}
{M^{\tiny\textcircled{\#}}_{i,j}} =
\frac{\det {\left( {M{M^\ast  }M + L{L^\ast  }} \right)}
\left( {i \to \left(MM^\ast\right)  e_j} \right) }
{\det {\left( {M{M^\ast  }M + L{L^\ast  }} \right)}},
\end{align}
where
$1\leq i,j \leq n$.
\end{theorem}

\begin{proof}
Since
$ {M{M^\ast  }M + L{L^\ast  }} $ is invertible,
we consider
$$\left( {M{M^\ast  }M + L{L^\ast  }} \right)x
=
\left(MM^\ast\right)  e_j$$
and get the solution
\begin{align*}
e_i^Tx =
\frac{\det {\left( {M{M^\ast  }M + L{L^\ast  }} \right)}
\left( {i \to \left(MM^\ast\right)  e_j} \right) }
{\det {\left( {M{M^\ast  }M + L{L^\ast  }} \right)}},
\end{align*}
in which $i, j=1,\ldots,n$.

It follows from (\ref{CM-000}) and
${M^{\tiny\textcircled{\#}}_{i,j}}=e_i^TM^{\tiny\textcircled{\#}}e_j$
that
we get (\ref{CM-5-2}).
\end{proof}

\begin{example}
\label{CM-Example-1}
Let
$M
=
\left[\begin{matrix}
1&2\\
0&0\\
\end{matrix}\right]$,
$L
=
\left[\begin{matrix}0
\\
1\end{matrix}\right]$
and
$b=
\left[\begin{matrix}
1\\1\end{matrix}\right]$.
It is easy to check that
$\R\left(L\right)=\N\left( {{M^\ast}} \right)$,
$
\left(I_n-M^{\tiny\textcircled{\#}}M \right)
L\left(L^\ast L\right)^{-1}
=
\left[\begin{matrix}-2\\1\end{matrix}\right]$,
$\left(L^\ast L\right)^{-1}  L^\ast   =
\left[\begin{matrix}0 & 1\end{matrix}\right]
$,
$G
=
\left[ {\begin{matrix}
1&2&0\\
0&0&1\\
0&1&0\\
\end{matrix}} \right]$,
$\det(G)=-1$
and
$
G^{-1}
=\left[ {\begin{matrix}
1&0&-2\\
0&0&1\\
0&1&0\\
\end{matrix}} \right]$. %=

By applying Lemma \ref{CM-lemma-M-1}, we have
$M^{\tiny\textcircled{\#}}
=
\left[\begin{matrix}1&0\\0&0\end{matrix}\right]$.
By applying Theorem \ref{CM-Theorem-1},
we get %$x_1=1$ and  $x_2=0$,
%so,
  the solution of (\ref{CM-1})
  is $x=M^{\tiny\textcircled{\#}}b=
\left[ {\begin{matrix}
1\\0\end{matrix}} \right]$.

For
$\det \left(
\left[ {\begin{matrix}
1&2&0\\
1&0&1\\
0&1&0\\
\end{matrix}} \right]\right)=-1$
and
$
\det \left(
\left[ {\begin{matrix}
1&1&0\\
0&1&1\\
0&0&0\\
\end{matrix}} \right]\right)=0$,
by applying Theorem \ref{CM-Theorem-3},
we get $x_1=\frac{-1}{-1}$ and  $x_2=\frac{0}{-1}$.
Therefore,
  the solution of (\ref{CM-1})
  is $x=
\left[ {\begin{matrix}
1\\0\end{matrix}} \right]$.

For
${\det {\left( {M{M^\ast  }M + L{L^\ast  }} \right)}}
=5$,
$\det {\left( {M{M^\ast  }M + L{L^\ast  }} \right)}
\left( {1 \to M{M^\ast  }b} \right)
=5
$
and
$\det {\left( {M{M^\ast  }M + L{L^\ast  }} \right)}
\left( {2 \to M{M^\ast  }b} \right)
=0$,
by applying Theorem \ref{CM-Theorem-4},
we get  $x_1=\frac{5}{5}$ and  $x_2=\frac{0}{5}$.
Therefore,
the solution of (\ref{CM-1})
  is $x=
\left[ {\begin{matrix}
1\\0\end{matrix}} \right]$.
\end{example}

%%%%%%%%%%%%%%%%%%%%%%%%%%%%%%%%%%%%%%%%%%%%%%%%%%%%%%%%%%%%%%%%%%
%%%%%%%%%%%%%%%%%%%%%%%%%%%%%%%%%%%%%%%%%%%%%%%%%%%%%%%%%%%%%%%%%%
\section*{Disclosure statement}
No potential conflict of interest was reported by the authors.
\section*{Funding}
The first
 author was supported partially by
Guangxi Natural Science Foundation [grant number 2018GXNSFAA138181],
the National Natural Science Foundation of China
[grant number 61772006]
and
the Special Fund for Bagui Scholars of Guangxi [grant number 2016A17].
The second
 author was supported partially by
the National Natural Science Foundation of China
[grant number  11361009]
and
High Level Innovation Teams and Distinguished Scholars in Guangxi Universities
[grant number GUIJIAOREN201642HAO]
%%%%%%%%%%%%%%%%%%%%%%%%%%%%%%%%%%%%%%%%%%%%%%%%%%%%%%%%%%%%%%%%%%
%%%%%%%%%%%%%%%%%%%%%%%%%%%%%%%%%%%%%%%%%%%%%%%%%%%%%%%%%%%%%%%%%%

\end{document}